\documentclass[a4paper,12pt]{article}

\usepackage{amsmath}
\usepackage{amsthm}
\usepackage{amssymb}

\usepackage{latexsym}
\usepackage{array}
\usepackage{enumerate}

\numberwithin{equation}{section}

\newcommand{\rhom}{R{\mathcal{H}}om}

\newcommand{\CC}{\mathbb{C}}

\newcommand{\RR}{\mathbb{R}}

\newcommand{\ZZ}{\mathbb{Z}}
\newcommand{\D}{\mathcal{D}}

\newcommand{\F}{\mathcal{F}}
\newcommand{\G}{\mathcal{G}}
\newcommand{\LL}{\mathcal{L}}
\newcommand{\M}{\mathcal{M}}

\newcommand{\sho}{\mathcal{O}}

\newcommand{\PP}{{\mathbb P}}

\newcommand{\Vspace}{\vspace{3ex}}

\renewcommand{\dim}{{\rm dim}}

\newcommand{\Vol}{{\rm Vol}}

\newcommand{\e}{\varepsilon}

\newcommand{\id}{{\rm id}}

\newcommand{\Sol}{{\rm Sol}}

\newcommand{\Db}{{\bf D}^{b}}
\newcommand{\Dbc}{{\bf D}_{c}^{b}}

\newcommand{\tl}[1]{\widetilde{#1}}

\newcommand{\simto}{\overset{\sim}{\longrightarrow}}

\newcommand{\dsum}{\displaystyle \sum}

\newcommand{\CF}{{\rm CF}}

\newtheorem{definition}{Definition}[section]
\newtheorem{theorem}[definition]{Theorem}
\newtheorem{proposition}[definition]{Proposition}
\newtheorem{lemma}[definition]{Lemma}

\newtheorem{example}[definition]{Example}

\title{Monodromy at infinity of $A$-hypergeometric 
functions and toric compactifications 
}

\author{Kiyoshi \textsc{Takeuchi}
\footnote{Institute of Mathematics, University  of 
Tsukuba, 1-1-1, Tennodai, 
Tsukuba, Ibaraki, 305-8571, Japan. 
E-mail: takemicro@nifty.com }}

\date{}

\begin{document}

\maketitle

\begin{abstract}
We study $A$-hypergeometric functions 
introduced by Gelfand-Kapranov-Zelevinsky 
\cite{G-K-Z-1} and 
prove a formula for the eigenvalues of 
their monodromy automorphisms 
defined by the analytic continuaions 
along large loops 
contained in complex lines parallel to 
the coordinate axes. The method of toric 
compactifications introduced in \cite{L-S} 
and \cite{M-T-3} will be used to prove 
our main theorem. 
\end{abstract}

\section{Introduction}\label{sec:1}

The theory of $A$-hypergeometric systems 
introduced by Gelfand-Kapranov-Zelevinsky 
\cite{G-K-Z-1} is an ultimate generalization 
of that of classical hypergeometric differential 
equations. As in the case of hypergeometric 
equations, the holomorphic solutions to 
$A$-hypergeometric systems admit power series 
expansions \cite{G-K-Z-1} and integral 
representations \cite{G-K-Z-2}. Moreover 
this theory has very deep connections with 
other fields of mathematics, such as 
toric varieties, projective duality, 
period integrals, mirror symmetry and 
combinatorics. Also from the viewpoint of 
$\D$-module theory (\cite{Dimca}, 
\cite{H-T-T} and \cite{K-S} etc.), 
$A$-hypergeometric 
$\D$-modules are very 
elegantly constructed 
in \cite{G-K-Z-2}. For the recent 
development of this subject see 
\cite{S-W} etc. However, to the best of 
our knowledge, it seems that the monodromy 
representations of their solutions i.e. 
$A$-hypergeometric functions are not 
completely studied yet. One of the most 
successful attempts toward the understanding 
of these monodromy representations would 
be Borisov-Horja's Mellin-Barnes type connection 
formulas for $A$-hypergeometric functions 
in \cite{B-H} and \cite{Horja}. 
In this paper, we study 
the monodromy representations in the light 
of the theory of $\D$-modules and 
constructible sheaves. In particular, 
we give a formula for the characteristic 
polynomials of the monodromy automorphisms 
of $A$-hypergeometric functions obtained 
by the analytic continuations along large loops 
contained in complex lines parallel to 
the coordinate axes. Namely we study 
the monodromy at infinity of 
$A$-hypergeometric functions (as for 
the topological monodromy at 
infinity in the theory of 
polynomial maps, see \cite{L-S} 
and \cite{M-T-3} etc.).  Our result 
are general and will be described only by the 
configulation $A$ and the parameter 
$\gamma \in \CC^n$ which are used 
to define the $A$-hypergeometric 
system. 

 In order to give the precise statement 
of our theorem, let us recall the definition 
of $A$-hypergeometric systems in 
\cite{G-K-Z-1} and \cite{G-K-Z-2}. 
Let $A=\{ a(1), a(2), \ldots , 
a(m)\} \subset \ZZ^{n-1}$ be a finite subset 
of the lattice $\ZZ^{n-1}$. Assume that 
$A$ generate $\ZZ^{n-1}$ as an affine lattice. 
Then the convex hull $Q$ of $A$ in 
$\RR^{n-1}_v = \RR \otimes_{\ZZ} \ZZ^{n-1}$ 
is an $(n-1)$-dimensional 
polytope. 
For $j=1,2, \ldots , m$ 
set $\tl{a(j)}:=(a(j),1) 
\in \ZZ^{n}=\ZZ^{n-1} 
\oplus \ZZ$ and consider the 
$n \times m$ integer matrix
\begin{equation}
\tl{A}:=\begin{pmatrix}
 {^t\tl{a(1)}} & {^t\tl{a(2)}} &\cdots& 
{^t\tl{a(m)}} \end{pmatrix}=(a_{ij}) \in M(n, m, \ZZ)
\end{equation}
whose $j$-th column is ${^t\tl{a(j)}}$. 
Then the GKZ hypergeometric 
system on $X=\CC^A=\CC_z^{m}$ 
associated with $A \subset \ZZ^{n-1}$ 
and a parameter $\gamma \in 
\CC^{n}$ is 
\begin{gather}
\left(\sum_{j=1}^{m} a_{ij}
z_j\frac{\partial}{\partial 
z_j}-\gamma_i\right)
 f(z)=0 \hspace{5mm} (1\leq i\leq n), \\
\left\{ \prod_{\mu_j>0} \left(\frac{\partial}{\partial 
z_j}\right)^{\mu_j} -\prod_{\mu_j<0} 
\left(\frac{\partial}{\partial 
z_j}\right)^{-\mu_j} \right\} f(z)=0 
\hspace{5mm} (\mu \in {\rm Ker} \tl{A} \cap 
\ZZ^{m}\setminus \{0\})
\end{gather}
(see \cite{G-K-Z-1} and \cite{G-K-Z-2}). 
Let $\D_{X}$ be the sheaf 
of differential operators with 
holomorphic coefficients 
on $X=\CC_z^{m}$ and set
\begin{eqnarray}
P_i&:=&\sum_{j=1}^{m} a_{ij}z_j
\frac{\partial}{\partial z_j}-\gamma_i 
\hspace{5mm}(1\leq i\leq n),\\
\square_{\mu} 
&:=&\prod_{\mu_j>0}\left(\frac{\partial}{\partial 
z_j}\right)^{\mu_j} -\prod_{\mu_j<0} \left(
\frac{\partial}{\partial 
z_j}\right)^{-\mu_j}\hspace{5mm} (\mu 
\in {\rm Ker} \tl{A} \cap 
\ZZ^{m}\setminus \{0\}).
\end{eqnarray}
Then the coherent $\D_{X}$-module
\begin{equation}
\M_{A, \gamma}=
\D_{X} /( \sum_{1 \leq i \leq n} 
\D_{X} P_i + \sum_{\mu \in {\rm Ker} \tl{A} 
\cap \ZZ^{m}\setminus 
\{0\}}\D_{X} \square_{\mu})
\end{equation}
which corresponds to 
the above system is holonomic 
and its solution complex 
\begin{equation}
{\rm Sol}(\M_{A, \gamma})=\rhom_{\D_X}
(\M_{A, \gamma}, \sho_X) 
\end{equation}
is a local system on an open dense 
subset of $X$. It is well-known 
after \cite{G-K-Z-1} and 
\cite{G-K-Z-2} that the singular 
locus of ${\rm Sol}(\M_{A, \gamma})$ 
is described by the $A$-discriminant 
varieties studied precisely by \cite{G-K-Z} 
(see also \cite{M-T-1} etc.). 

 \par \indent Let us fix an integer 
$j_0 \in \ZZ$ such that 
$1 \leq j_0 \leq m$ and for 
$(c_1, c_2, \ldots , c_{j_0-1}, 
c_{j_0+1}, \ldots , c_m) \in \CC^{m-1}$ 
consider a complex line $L \simeq 
\CC$ in $X=\CC^m_z$ defined by 
\begin{equation}
L=\{ z=(z_1, z_2, \ldots , z_m)  \in \CC^m \ | \ 
z_j=c_j \ \text{for} \ j \not= j_0\} \subset X=\CC^m_z. 
\end{equation}
Since $L$ is a line parallel to the 
$j_0$-th axis $\CC_{z_{j_0}}$ of $X=\CC^m_z$, 
we naturally identify 
it with $\CC_{z_{j_0}}$. For simplicity, 
we denote the $j_0$-th coordinate 
function $z_{j_0} : X=\CC^m_z 
\longrightarrow \CC$ by $s$. Then 
it is well-known that if 
$(c_1, c_2, \ldots , c_{j_0-1}, 
c_{j_0+1}, \ldots , c_m) \in \CC^{m-1}$ 
is generic there exists 
a finite subset $S_L \subset L \simeq 
\CC_s$ such that ${\rm Sol} (\M_{A, \gamma})|_L$ 
is a local system on $L \setminus S_L$. 
Let us take such a line $L$ in 
$X=\CC^m_z$ and a point $s_0 \in 
L \simeq \CC_s$ in $L$ such that 
$|s_0| > \max_{s\in S_L}|s|$.  
Then we obtain a monodromy automorphism 
\begin{equation}
{\rm Sol}(\M_{A, \gamma})_{s_0} 
\simto  
{\rm Sol}(\M_{A, \gamma})_{s_0} 
\end{equation}
defined by the analytic continuation of the 
sections of ${\rm Sol}(\M_{A, \gamma})|_L$ 
along the path 
\begin{equation}
C_{s_0}=\{ s_0\exp( \sqrt{-1} \theta ) 
\ | \ 0 \leq \theta \leq 2 \pi \}
\end{equation}
in $L \simeq \CC_s$. Since the characteristic 
polynomial of this automorphism does 
not depend on $L$ and $s_0 \in L$,  
we denote it simply by 
$\lambda_{j_0}^{\infty}(t) \in \CC [t]$. 
We call $\lambda_{j_0}^{\infty}(t)$ 
the characteristic polynomial of 
the $j_0$-th monodromy at infinity 
of the $A$-hypergeometric functions 
${\rm Sol}(\M_{A, \gamma})$. 
According to the fundamental result of \cite{G-K-Z-1},  
if $\gamma \in \CC^n$ is 
generic (i.e. non-resonant in the sense 
of \cite[Section 2.9]{G-K-Z-2}),  
the rank of the local system 
${\rm Sol}(\M_{A, \gamma})$ 
is equal to the normalized $(n-1)$-dimensional 
volume $\Vol_{\ZZ}(Q) \in \ZZ$ 
of $Q$ with respect to the lattice 
$\ZZ^{n-1}$. See also Saito-Sturmfels-Takayama 
\cite{S-S-T} for the related results. 
Therefore, the degree of the 
the characteristic polynomial 
$\lambda_{j_0}^{\infty}(t)$ is 
$\Vol_{\ZZ}(Q)$. In order to give a formula 
for $\lambda_{j_0}^{\infty}(t) \in 
\CC [t]$ we shall prepare some notations. 
First, we set $\alpha =\gamma_n, \beta_1=
-\gamma_1-1, \beta_2=-\gamma_2-1, \ldots , 
\beta_{n-1}=-\gamma_{n-1}-1$ and 
$\beta =(\beta_1, \beta_2, \ldots , 
\beta_{n-1})\in \CC^{n-1}$ (see 
\cite[Theorem 2.7]{G-K-Z-2}). Next 
let $\Delta_1, \Delta_2, \ldots , 
\Delta_k$ be the $(n-2)$-dimensional 
faces i.e. the facets of $Q$ such that 
$a(j_0) \notin \Delta_r$ 
($r=1,2, \ldots , k$). Then for 
each $r=1,2, \ldots , k$ there exists 
a unique primitive vector $u^r \in 
\ZZ^{n-1}\setminus \{ 0\}$ such that 
\begin{equation}
\Delta_r=\{ v\in Q \ | \ 
\langle u^r , v \rangle =\min_{w\in Q} 
\langle u^r, w \rangle \}. 
\end{equation} 
Let us set 
\begin{eqnarray}
h_r & = & \min_{w\in Q} 
\langle u^r, w \rangle = 
\langle u^r, \Delta_r \rangle \in \ZZ , 
\\ 
d_r & = & \langle u^r, 
a(j_0) \rangle -h_r 
\in \ZZ . 
\end{eqnarray}
Since $-u^r \in \ZZ^{n-1} \subset 
\RR^{n-1}$ is the primitive outer 
conormal vector of the facet 
$\Delta_r \prec Q$ of $Q$ and 
we have 
\begin{equation}
d_r=\langle -u^r, w-a(j_0) \rangle 
\end{equation}
for any $w \in \Delta_r$, the integer 
$d_r$ is the lattice distance of 
the point $a(j_0) \in Q$ from $\Delta_r$. 
In particular, 
we have $d_r > 0$. Finally we set 
\begin{equation}
\delta_r=\alpha h_r + 
\langle \beta , u^r \rangle \in \CC
\end{equation}
for $r=1,2, \ldots , k$. 
Then we obtain the following theorem. 

\begin{theorem}\label{TH1} 
Assume that $\gamma \in \CC^n$ is 
non-resonant in the sense 
of \cite[Section 2.9]{G-K-Z-2}. Then the 
characteristic polynomial 
$\lambda_{j_0}^{\infty}(t)$ of 
the $j_0$-th monodromy at infinity 
of ${\rm Sol}(\M_{A, \gamma})$ is given by 
\begin{equation}
\lambda_{j_0}^{\infty}(t)= 
\prod_{r=1}^k \left\{ t^{d_r} -
\exp (-2\pi \sqrt{-1}\delta_r) 
\right\}^{\Vol_{\ZZ}(\Delta_r)}, 
\end{equation}
where $\Vol_{\ZZ}(\Delta_r) \in \ZZ$ 
is the normalized $(n-2)$-dimensional 
volume of $\Delta_r$. 
\end{theorem}
For the proof of this theorem, we will 
use some sheaf-theoretical methods 
such as nearby and constructible 
sheaves, and a toric compactification of 
the algebraic torus 
$T=(\CC^*)^n \subset (\CC^*)^{n-1} 
\times L$ similar to the one 
used in the study of 
topological monodromy at infinity 
of polynomial maps in \cite{L-S} 
and \cite{M-T-3}. As well as 
Bernstein-Khovanskii-Kushnirenko's 
theorem \cite{Khovanskii}, the general 
results on nearby cycle sheaves 
of local systems obtained 
in \cite[Section 5]{M-T-2} will 
play an important role in the proof. 

\Vspace 
\par \indent Finally, the author 
would like to express his hearty 
gratitude to Professors Y. Haraoka, Y. Matsui 
and N. Takayama for some very useful discussions 
on this subject.

\section{Preliminary 
notions and results}\label{sec:2}

In this section, we introduce basic 
notions and results which will be used 
in the proof of our main theorem. We 
essentially follow the terminology of 
\cite{Dimca}, \cite{H-T-T} and \cite{K-S}. 
For example, for a topological space $X$ we 
denote by $\Db(X)$ the derived 
category whose objects are bounded complexes 
of sheaves of $\CC_X$-modules on $X$.

\begin{definition}\label{dfn:2-1}
Let $X$ be an algebraic variety over $\CC$. Then
\begin{enumerate}
\item We say that a sheaf $\F$ on $X$ 
is constructible if there exists a 
stratification $X=\bigsqcup_{\alpha} 
X_{\alpha}$ of $X$ such 
that $\F|_{X_{\alpha}}$ is a locally 
constant sheaf of finite rank for any $\alpha$.
\item We say that an object $\F$ 
of $\Db(X)$ is constructible if the 
cohomology sheaf $H^j(\F)$ of 
$\F$ is constructible for any $j \in \ZZ$. We 
denote by $\Dbc(X)$ the full 
subcategory of $\Db(X)$ consisting of 
constructible objects $\F$.
\end{enumerate}
\end{definition}

Recall that for any morphism 
$f \colon X \longrightarrow Y$ of algebraic 
varieties over $\CC$ there exists a functor
\begin{equation}
Rf_* \colon \Db(X) \longrightarrow \Db(Y)
\end{equation}
of direct images. 
This functor preserves the 
constructibility and we obtain also a functor
\begin{equation}
Rf_* \colon \Dbc(X) 
\longrightarrow \Dbc(Y).
\end{equation}
For other basic 
operations $Rf_!$, $f^{-1}$, $f^!$ etc. in derived 
categories, see \cite{K-S} for the detail.

Next we introduce the 
notion of constructible functions.

\begin{definition}\label{dfn:2-2}
Let $X$ be an algebraic variety 
over $\CC$ and $G$ an abelian group. Then we 
say a $G$-valued function $\rho \colon X 
\longrightarrow G$ on $X$ is 
constructible if there exists a 
stratification $X=\bigsqcup_{\alpha} 
X_{\alpha}$ of $X$ such that 
$\rho|_{X_{\alpha}}$ is constant 
for any $\alpha$. We denote by $\CF_G(X)$ 
the abelian group of $G$-valued 
constructible functions on $X$.
\end{definition}

Let $\CC(t)^*=\CC(t) \setminus \{0\}$ be 
the multiplicative group of the 
function field $\CC(t)$ of the scheme 
$\CC$. In this paper, we consider 
$\CF_G(X)$ only for $G=\ZZ$ or $\CC(t)^*$. 
For a $G$-valued constructible 
function $\rho \colon X 
\longrightarrow G$, by taking a stratification 
$X=\bigsqcup_{\alpha}X_{\alpha}$ of $X$ 
such that $\rho|_{X_{\alpha}}$ is 
constant for any $\alpha$ as above, we set
\begin{equation}
\int_X \rho :=\dsum_{\alpha}
\chi(X_{\alpha}) \cdot \rho(x_{\alpha}) \in G,
\end{equation}
where $x_{\alpha}$ is a reference 
point in $X_{\alpha}$. Then we can easily 
show that $\int_X\rho \in G$ 
does not depend on the choice of the 
stratification $X=\bigsqcup_{\alpha} 
X_{\alpha}$ of $X$. Hence we obtain a 
homomorphism
\begin{equation}
\int_X \colon \CF_G(X) \longrightarrow G
\end{equation}
of abelian groups. For $\rho 
\in \CF_G(X)$, we call $\int_X \rho \in G$ the 
topological (Euler) integral of 
$\rho$ over $X$. More generally, for any 
morphism $f \colon X \longrightarrow 
Y$ of algebraic varieties over $\CC$ 
and $\rho \in \CF_G(X)$, we 
define the push-forward $\int_f \rho \in 
\CF_G(Y)$ of $\rho$ by
\begin{equation}
\left( \int_f \rho \right) (y):=
\int_{f^{-1}(y)} \rho
\end{equation}
for $y \in Y$. This defines a homomorphism
\begin{equation}
\int_f \colon \CF_G(X) \longrightarrow \CF_G(Y)
\end{equation}
of abelian groups.

Among various operations in derived 
categories, the following nearby and 
vanishing cycle functors 
introduced by Deligne will be frequently used in 
this paper (see \cite[Section 4.2]{Dimca} 
for an excellent survey of this 
subject). Let $f: X \longrightarrow \CC$ 
be a regular function on an algebraic 
variety $X$ over $\CC$  
and set $X_0 =\{ x\in X \ | \ 
f(x)=0 \}$. Then there exist functors 
\begin{equation}
\psi_f, \ \varphi_f \colon 
\Dbc(X) \longrightarrow \Dbc(X_0).
\end{equation}
which are called 
the nearby and vanishing cycle 
functors of $f$ respectively. 
As we see in the next proposition, 
the nearby cycle functor $\psi_f$ 
generalizes the classical notion of 
Milnor fibers. First, let us recall the 
definition of Milnor fibers and 
Milnor monodromies over singular varieties 
(see for example \cite{Takeuchi} 
for a review on this subject). Let $X$ be a 
subvariety of $\CC^m$ and $f 
\colon X \longrightarrow \CC$ a non-constant 
regular function on $X$. Namely 
we assume that there exists a polynomial 
function $\tl{f} \colon \CC^m 
\longrightarrow \CC$ on $\CC^m$ such that 
$\tl{f}|_X=f$. For simplicity, 
assume also that the origin $0 \in \CC^m$ is 
contained in $X_0=\{x \in X \ 
|\ f(x)=0\}$. Then the following lemma is 
well-known (see for example 
\cite{Milnor} and 
\cite[Definition 1.4]{Massey}).

\begin{lemma}\label{lem:2-5}
For sufficiently small $\e >0$, 
there exists $\eta_0 >0$ with $0<\eta_0 \ll 
\e$ such that for $0 < \forall 
\eta <\eta_0$ the restriction of $f$:
\begin{equation}
X \cap B(0;\e) \cap \tl{f}^{-1}
(D(0;\eta) \setminus \{0\}) \longrightarrow 
D(0;\eta) \setminus \{0\}
\end{equation}
is a topological fiber bundle 
over the punctured disk $D(0;\eta) \setminus 
\{0\}=\{ z \in \CC \ |\ 0<|z|<\eta\}$, 
where $B(0;\e)$ is the open ball in 
$\CC^m$ with radius $\e$ centered at the origin.
\end{lemma}

\begin{definition}\label{dfn:2-6}
A fiber of the above fibration is 
called the Milnor fiber of $f 
\colon X\longrightarrow \CC$ at $0 
\in X$ and we denote it by $F_0$.
\end{definition}
Similarly, for $x \in X_0$ we define 
 the Milnor fiber $F_x$ of $f$ at $x$. 

\begin{proposition}{\rm 
\bf(\cite[Proposition 
4.2.2]{Dimca})}\label{prp:2-7} 
For any $x \in X_0$ and $j \in \ZZ$ there 
exists a natural isomorphism
\begin{equation}
H^j(F_x;\CC) \simeq H^j(\psi_f(\CC_X))_x. 
\end{equation}
\end{proposition}

By this proposition, we can study 
the cohomology groups $H^j(F_x;\CC)$ of 
the Milnor fiber $F_x$ by using sheaf 
theory. Recall also that in the above 
situation, as in 
the case of polynomial functions over 
$\CC^n$ (see \cite{Milnor}), we can 
define the Milnor monodromy operators
\begin{equation}
\Phi_{j,x} \colon H^j(F_x ;\CC) 
\overset{\sim}{\longrightarrow} H^j(F_x ;\CC)
 \qquad \ 
(j=0,1,\ldots)
\end{equation}
and the zeta-function
\begin{equation}
\zeta_{f,x}(t):=\prod_{j=0}^{\infty} 
\det \left( \id -t\Phi_{j,x} \right)^{(-1)^j}
\end{equation}
associated with it. Since the above 
product is in fact finite, 
$\zeta_{f,x}(t)$ is a rational 
function of $t$ and its degree in $t$ is the 
topological Euler characteristic 
$\chi(F_x)$ of the Milnor fiber $F_x$. 
For the explicit formulas of 
$\zeta_{f,x}(t)$ and $\chi (F_x)$, 
see \cite{Kushnirenko}, 
\cite{Milnor}, \cite{Varchenko} and 
\cite{M-T-2} etc. 
This classical notion of Milnor monodromy 
zeta functions can be also generalized 
as follows.

\begin{definition}\label{dfn:2-8}
Let $f \colon X \longrightarrow \CC$ be 
a non-constant regular function on 
$X$ and $\F \in \Dbc(X)$. Set $X_0 
:=\{x\in X\ |\ f(x)=0\}$. Then there 
exists a monodromy automorphism
\begin{equation}
\Phi(\F) \colon \psi_f(\F) \simto \psi_f(\F)
\end{equation}
of $\psi_f(\F)$ in $\Dbc(X_0)$ 
(see \cite[Section 4.2]{Dimca}). We 
define a $\CC(t)^*$-valued constructible 
function $\zeta_f(\F) \in 
\CF_{\CC(t)^*}(X_0)$ on $X_0$ by
\begin{equation}
\zeta_{f,x}(\F)(t):=\prod_{j \in \ZZ} 
\det \left\{ \id -t\Phi(\F)_{j,x} 
\right\}^{(-1)^j}
\end{equation}
for $x \in X_0$, where $\Phi(\F)_{j,x} 
\colon (H^j(\psi_f(\F)))_x \simto 
(H^j(\psi_f(\F)))_x$ is the stalk at 
$x \in X_0$ of the sheaf homomorphism
\begin{equation}
\Phi(\F)_j \colon H^j(\psi_f(\F)) 
\simto H^j(\psi_f(\F))
\end{equation}
induced by $\Phi(\F)$.
\end{definition}

The following proposition will 
play an important role in the proof of 
our main theorem. 
For the proof, see for example, 
\cite[p.170-173]{Dimca} and \cite{Schurmann}.

\begin{proposition}\label{prp:2-9}
Let $\pi \colon Y \longrightarrow X$ be 
a proper morphism of algebraic 
varieties over $\CC$ and $f \colon X 
\longrightarrow \CC$ a non-constant 
regular function on $X$. Set $g:=f 
\circ \pi \colon Y \longrightarrow \CC$, 
$X_0:=\{x\in X\ |\ f(x)=0\}$ and 
$Y_0:=\{y\in Y\ |\ g(y)=0\}=\pi^{-1}(X_0)$. 
Then for any $\G\in \Dbc(Y)$ we have
\begin{equation}
\int_{\pi|_{Y_0}} \zeta_g(\G) =\zeta_f(R\pi_*\G)
\end{equation}
in $\CF_{\CC(t)^*}(X_0)$, where
\begin{equation}
\int_{\pi|_{Y_0}}\colon \CF_{\CC(t)^*}
(Y_0) \longrightarrow 
\CF_{\CC(t)^*}(X_0)
\end{equation}
is the push-forward of $\CC(t)^*$-valued 
constructible functions by 
$\pi|_{Y_0} \colon Y_0 \longrightarrow X_0$.
\end{proposition}

For the proof of the following 
propositions, see 
\cite[Proposition 5.2 
and 5.3]{M-T-2}.  

\begin{proposition}\label{prp:5-2}
(\cite[Proposition 5.2 (iii)]{M-T-2}) 
Let $\LL$ be a local system of 
rank $r>0$ on $\CC^*=\CC\setminus \{0\}$. 
Denote by $A \in GL_r(\CC)$ the 
monodromy matrix of $\LL$ along the loop 
$\{e^{i\theta}\ |\ 0 \leq 
\theta \leq 2\pi\}$ in $\CC^*$, which is defined 
up to conjugacy. Let $j 
: \CC^* \hookrightarrow \CC$ be the 
inclusion and $h$ a 
function on $\CC$ defined by $h(z)=z^m$ ($m\in 
\ZZ_{>0}$) for $z \in \CC$. Then we have
\begin{equation}\label{eq:5-9}
\zeta_{h,0}(j_!\LL)(t) 
=\det (\id -t^mA) \in \CC(t)^*.
\end{equation}
\end{proposition}

\begin{proposition}\label{prp:5-3}
(\cite[Proposition 5.3]{M-T-2}) 
Let $\LL$ be a local system on 
$(\CC^*)^k$ for $k \geq 2$ 
and $j \colon (\CC^*)^k 
\hookrightarrow \CC^k$ 
the inclusion. Let $h : \CC^k 
\longrightarrow \CC$ be a 
function defined by $h(z)=z_1^{m_1}z_2^{m_2}\cdots 
z_k^{m_k} \not\equiv 1$ 
($m_i \in \ZZ_{\geq 0}$) for 
$z\in \CC^k$. Then 
the monodromy zeta function 
$\zeta_{h,0}(j_!\LL)(t)$ of 
$j_!\LL \in \Dbc(\CC^k)$ at $0 
\in \CC^k$ is $1 
\in \CC(t)^*$.
\end{proposition}

\section{Monodromy at infinity of 
$A$-hypergeometric functions}\label{sec:3}

In this section, we inherit the notations 
and the situation in the introduction 
and prove our main theorem (Theorem \ref{TH1}). 
First, let us recall the definition of 
the non-resonance of the parameter 
$\gamma \in \CC^n$ introduced in 
\cite[Section 2.9]{G-K-Z-2}. Let $K$ 
be a convex cone in $\RR^n$ generated by 
the vectors $(a(1), 1), (a(2),1), \ldots 
, (a(m),1) \in \ZZ^n=\ZZ^{n-1}\oplus \ZZ$. 
For each face $\Gamma \prec K$ of $K$ 
denote by ${\rm Lin}(\Gamma) \simeq 
\CC^{\dim \Gamma}$ the $\CC$-linear span 
of $\Gamma$ in $\CC^n$. 

\begin{definition}
(\cite[Section 2.9]{G-K-Z-2}) We say 
that the parameter 
$\gamma \in \CC^n$ is non-resonant if for 
any face $\Gamma \prec K$ of codimension 
one we have 
$\gamma \notin \ZZ^n+ {\rm Lin}(\Gamma)$. 
\end{definition}
By the fundamental result of \cite{G-K-Z-1},  
if $\gamma \in \CC^n$ is non-resonant the 
generic rank of $\Sol(\M_{A, \gamma})$ 
is equal to the normalized $(n-1)$-dimensional 
volume $\Vol_{\ZZ}(Q) \in \ZZ$ 
of $Q$. Therefore, for any 
$j_0 \in \ZZ$ such that $1 \leq j_0 \leq m$ 
the degree of the 
the characteristic polynomial 
$\lambda_{j_0}^{\infty}(t)$ is 
$\Vol_{\ZZ}(Q)$. 

\begin{theorem}
Assume that $\gamma \in \CC^n$ is 
non-resonant. Then the 
characteristic polynomial 
$\lambda_{j_0}^{\infty}(t)$ of 
the $j_0$-th monodromy at infinity 
of $\Sol(\M_{A, \gamma})$ is given by 
\begin{equation}
\lambda_{j_0}^{\infty}(t)= 
\prod_{r=1}^k \left\{ t^{d_r} -
\exp (-2\pi \sqrt{-1}\delta_r) 
\right\}^{\Vol_{\ZZ}(\Delta_r)}, 
\end{equation}
where $\Vol_{\ZZ}(\Delta_r) \in \ZZ$ 
is the normalized $(n-2)$-dimensional 
volume of $\Delta_r$. 
\end{theorem}
\begin{proof}
Let 
\begin{equation}
L=\{ z=(z_1,z_2, \ldots , z_m) 
\in \CC^m \ | \ 
z_j=c_j \ \text{for} \ j \not= j_0\} 
\subset X=\CC^m_z . 
\end{equation}
($(c_1, c_2, \ldots , c_{j_0-1}, 
c_{j_0+1}, \ldots , c_m) \in \CC^{m-1}$) be the 
defining equation of $L \simeq \CC_s$ 
in $X=\CC^m_z$ and define a Laurent polynomial 
$p$ on $(\CC^*)_x^{n-1} \times L 
\simeq (\CC^*)_x^{n-1} \times \CC_s$ by 
\begin{equation}
p(x,s)=sx^{a(j_0)}+\sum_{j \not= j_0}
c_jx^{a(j)}. 
\end{equation}
Denote by $\tl{P}$ the convex hull of 
$(a(j_0), 1) \sqcup \{ (a(j), 0) \ | \ 
j \not= j_0 \}$ in $\RR^n_{\tl{v}}
=\RR^{n-1}_v \oplus \RR$. 
We may assume that 
$c_j \not= 0$ for any $j \not= j_0$ 
and the Newton polytope of 
$p(x,s)$ is $\tl{P}$. 
However note that the 
dimension of the polytope $\tl{P}$ 
is not necessarily equal to $n$. Let 
$U$ be an open subset of 
$(\CC^*)_x^{n-1} \times L$ defined by 
$U=\{ (x, s) \in (\CC^*)^{n-1} \times L 
\ | \ p(x,s) \not= 0\}$ and $\pi =s: 
U \longrightarrow L \simeq \CC$ 
the restriction of the second projection 
$(\CC^*)^{n-1} \times L \longrightarrow 
L$ to $U$. Define a local system $\LL$ of 
rank one on $U$ by 
\begin{equation}
\LL= \CC \ p(x,s)^{\alpha} x_1^{\beta_1} 
x_2^{\beta_2} \cdots x_{n-1}^{\beta_{n-1}}. 
\end{equation}
Then by \cite[page 270, line 9-10]{G-K-Z-2} 
we have an isomorphism 
\begin{equation}
\Sol(\M_{A, \gamma})|_L \simeq 
R\pi_! \LL [n-1]
\end{equation}
in $\Dbc (L)$. Let $j: L \simeq \CC_s 
\hookrightarrow \CC_s \sqcup \{ \infty \} 
=\PP^1$ be the embedding and $h(s)=\frac{1}{s}$ 
the holomorphic function defined on 
an neighborhood of $\infty$ in $\PP^1$ 
such that $\{ \infty \} =\{ h=0\}$. 
Then it suffices to show that the 
monodromy zeta function 
$\zeta_{h, \infty}(j_! R\pi_! \LL [n-1]) 
(t) \in \CC (t)^*$ of the constructible 
sheaf $j_! R\pi_! \LL [n-1] \in 
\Dbc (\PP^1)$ at $\infty \in \PP^1$ 
is given by 
\begin{equation}
\zeta_{h, \infty}(j_! R\pi_! \LL [n-1]) 
(t) = 
\prod_{r=1}^k \{1-
\exp (2\pi \sqrt{-1}\delta_r) t^{d_r} 
\}^{\Vol_{\ZZ}(\Delta_r)}. 
\end{equation}
Indeed, by the isomorphism 
$\CC_s^* \simeq \CC^*_h$, $h=\frac{1}{s}$ 
the sufficiently large circle 
\begin{equation}
C_{s_0}=\{ s_0\exp( \sqrt{-1} \theta ) 
\ | \ 0 \leq \theta \leq 2 \pi  \} 
\subset \CC^*_s \qquad \ 
(|s_0| >> 0) 
\end{equation}
of anti-clockwise direction 
is sent to the small one 
\begin{equation}
\tl{C}_{\frac{1}{s_0}}
=\{ \frac{1}{s_0}\exp(
- \sqrt{-1} \theta ) 
\ | \ 0 \leq \theta \leq 2 \pi \}
\subset \CC^*_h
\end{equation}
of clockwise direction. Let $T=
(\CC^*)_x^{n-1} \times \CC^*_s 
\simeq (\CC^*)^n$ be the open dense 
torus in $(\CC^*)_x^{n-1} \times L$ 
and $j^{\prime}: U \cap T \hookrightarrow T$ 
the inclusion. Then for the constructible 
sheaf $\F =j^{\prime}_! (\LL |_{U \cap T})$ 
on $T$ 
and the restriction $\pi^{\prime}: T 
\longrightarrow L \simeq \CC_s$ 
of the second projection $(\CC^*)_x^{n-1} 
\times L \longrightarrow L$ to $T$ 
we have 
\begin{equation}
\zeta_{h, \infty}(j_! R\pi_! \LL [n-1]) 
(t) = \zeta_{h, \infty}(j_! 
R\pi^{\prime}_! \F [n-1]) (t). 
\end{equation}
From now on, we shall construct a toric 
compactification $\overline{T}$ of $T$ 
such that the meromorphic extension 
of the coordinate function $s
=\pi^{\prime}:  T = 
(\CC^*)_x^{n-1} \times \CC^*_s 
\longrightarrow L \simeq \CC$ to 
$\overline{T}$ has no point of 
indeterminacy and induces a holomorphic 
map $g: \overline{T} \longrightarrow \PP^1$ 
such that $g \circ \iota = 
j \circ \pi^{\prime}$ for $\iota : 
T \hookrightarrow \overline{T}$. 
Let $\RR^n_{\tl{u}}$ be the dual 
vector space $(\RR^n_{\tl{v}})^*$ 
of $\RR^n_{\tl{v}} 
=\RR^{n-1}_v \oplus \RR$. We denote 
the dual lattice ${\rm Hom}_{\ZZ}(\ZZ^n , \ZZ)
\subset \RR^n_{\tl{u}}$ of $\ZZ^n \subset 
\RR^n_{\tl{v}}$ simply by $\ZZ^n$. 

\begin{definition}
For $\tl{u} \in \RR^n_{\tl{u}}$ 
we define the $\tl{u}$-part 
$p^{\tl{u}}(x,s)$ of the Laurent polynomial 
$p(x,s)$ by 
\begin{equation}
p^{\tl{u}}(x,s)
=\begin{cases}
\dsum_{\begin{subarray}{c}
j \not=j_0: \ (a(j),0) 
\\  \in \Gamma(\tl{P}; \tl{u}) 
\end{subarray}} 
c_j x^{a(j)}
 & \text{if $(a(j_0),1) 
\notin \Gamma(\tl{P}; \tl{u})$ }, 
\\
sx^{a(j_0)} + 
\dsum_{\begin{subarray}{c}
j \not=j_0: \ (a(j),0) 
\\  \in \Gamma(\tl{P}; \tl{u}) 
\end{subarray}} 
c_j x^{a(j)}
 & \text{if $(a(j_0),1) 
\in \Gamma(\tl{P}; \tl{u})$} , 
\end{cases}
\end{equation}
where we set 
\begin{equation}
\Gamma (\tl{P}; \tl{u}) 
=\{ \tl{v} \in \tl{P} \ | \ 
\langle \tl{u} , \tl{v} \rangle 
=\min_{ \tl{w} \in  \tl{P}} 
\langle \tl{u} , \tl{w}  \rangle 
\} \subset \tl{P}. 
\end{equation}
\end{definition} 

\begin{definition} 
We say that the Laurent polynomial $p(x,s)$ 
is non-degenerate if for any non-zero 
$\tl{u} \in \ZZ^n \subset \RR^n_{\tl{u}}$ 
the complex hypersurface 
\begin{equation}
\{ (x,s) \in T \ | \ p^{\tl{u}}(x,s)=0\} 
\end{equation} 
in $T \simeq (\CC^*)^n$ is smooth and reduced. 
\end{definition} 
Since $p(x,s)$ is non-degenerate for generic 
$(c_1, c_2, \ldots , c_{j_0-1}, 
c_{j_0+1}, \ldots , c_m) \in \CC^{m-1}$, for 
the calculation of $\zeta_{h, \infty}(j_! 
R\pi^{\prime}_! \F [n-1]) (t)$ we may 
assume that $p(x,s)$ is non-degenerate. Now 
we introduce an equivalence relation 
$ \sim $ of the dual vector space $\RR^n_{\tl{u}}$ 
of $\RR^n_{\tl{v}}$ defined by 
\begin{equation}
\tl{u} \sim \tl{u}^{\prime} 
\quad \Longleftrightarrow \quad 
\Gamma(\tl{P} ; \tl{u}) = 
\Gamma(\tl{P} ; \tl{u}^{\prime}). 
\end{equation} 
Then we obtain a decomposition $\RR^n_{\tl{u}}
=\bigsqcup \tau$ into equivalence classes $\tau$ and 
a subdivision $\Sigma_1=\{ \overline{\tau} \}$ 
of $\RR^n_{\tl{u}}$ into convex 
cones $\overline{\tau}$. 
Namely $\Sigma_1$ is the 
dual subdivision of $\tl{P} \subset 
\RR^n_{\tl{v}}$. 
Note that $\Sigma_1$ 
is not necessarily a fan in $\RR^n_{\tl{u}}$ 
since we do not assume $\dim \tl{P} =n$. 
Next for $\varepsilon_1, 
\varepsilon_2, \ldots , \varepsilon_n 
= \pm 1$ we set 
\begin{equation}
\sigma_{\varepsilon_1, \ldots , \varepsilon_n} 
= \{ \tl{u}=(\tl{u}_1, \tl{u}_2, \ldots , 
\tl{u}_n) \ | \ \varepsilon_i \tl{u}_i \geq 0 
\ \text{for} \ i=1,2, \ldots , n \} 
\end{equation} 
and consider the complete fan $\Sigma_2$ in 
$\RR^n_{\tl{u}}$ consisting of the cones 
$\sigma_{\varepsilon_1, \ldots , \varepsilon_n}$ 
($\varepsilon_1, 
\varepsilon_2, \ldots , \varepsilon_n 
= \pm 1$) 
and their faces. Let $\Sigma_0$ be the 
common subdivision of $\Sigma_1$ and 
$\Sigma_2$. Applying some more subdivisions 
to $\Sigma_0$ if necessary, we obtain a 
complete fan $\Sigma$ in $\RR^n_{\tl{u}}$ 
such that the toric variety $X_{\Sigma}$ 
associated to it is smooth and complete 
(see \cite{Fulton}, \cite{G-K-Z} and 
\cite{Oda} etc.). 
Recall that the algebraic torus $T$ acts 
on $X_{\Sigma}$ and there exists a 
natural bijection between the set of 
$T$-orbits in 
$X_{\Sigma}$ and that of the cones 
$\sigma \in \Sigma$ in $\Sigma$. 
For a cone $\sigma \in \Sigma$ denote 
by $T_{\sigma}$ the corresponding 
$T$-orbit in $X_{\Sigma}$. Then we obtain 
a decomposition $X_{\Sigma} = 
\sqcup_{\sigma \in \Sigma}T_{\sigma}$ 
of $X_{\Sigma}$ into $T$-orbits. Let 
$\iota : T \hookrightarrow X_{\Sigma}$ 
be the canonical embedding. 
We can show that the 
extension of $s=\pi^{\prime} 
: T \longrightarrow \CC 
$ to a meromorphic function on $X_{\Sigma}$ 
has no point of indeterminacy as follows. 
Indeed, for an $n$-dimensional cone 
$\sigma_0 \in \Sigma$ denote by 
 $\CC^n(\sigma_0)$ 
($\simeq \CC^n$) the smooth toric variety 
associated to the fan consisting of  
$\sigma_0$ and its faces. Then 
$\CC^n(\sigma_0)$ is an affine open 
subset of $X_{\Sigma}$ containing $T$ 
and $X_{\Sigma}$ 
is covered by such open subsets. 

\begin{lemma}
For any $n$-dimensional cone 
$\sigma_0 \in \Sigma$, the meromorphic 
extension of $s=\pi^{\prime} 
: T \longrightarrow \CC 
$ to $\CC^n(\sigma_0)$ has no point of 
indeterminacy. 
\end{lemma}
\begin{proof}
For an $n$-dimensional cone 
$\sigma_0 \in \Sigma$, let 
$\{ \tl{u}(\sigma_0)^1, 
\tl{u}(\sigma_0)^2, \ldots , 
\tl{u}(\sigma_0)^n \}$ be the 
$1$-skelton of $\sigma_0$ 
(i.e. the set of the 
primitive vectors on the 
edges of $\sigma_0$). Then 
on $\CC^n(\sigma_0)
\simeq \CC^n_y$ the meromorphic 
extension $g$ of $s=\pi^{\prime} 
: T \longrightarrow \CC$ has the form 
\begin{equation} 
g(y)=y_1^{k_1}y_2^{k_2} \cdots 
y_n^{k_n}, 
\end{equation} 
where we set 
\begin{equation} 
k_i= 
\langle \tl{u}(\sigma_0)^i, (0, \ldots , 0,1) 
\rangle \in \ZZ 
\end{equation} 
for $i=1,2, \ldots , n$.  Since $\Sigma$ 
is a subdivision of $\Sigma_2$, 
we have 
\begin{equation} 
k_i \geq 0 \quad (i=1,2, \ldots , n) 
\qquad \text{or} \qquad 
k_i \leq 0 \quad (i=1,2, \ldots , n). 
\end{equation} 
\end{proof}
By this lemma, 
there exists a holomorphic map 
$g: X_{\Sigma}=\overline{T} \longrightarrow 
\PP^1$ such that $g \circ \iota = 
j \circ \pi^{\prime}$. Since $g$ is proper, we 
thus obtain an isomorphism 
\begin{equation} 
j_! R\pi^{\prime}_! \F \simeq 
Rg_* \iota_! \F 
\end{equation}
in $\Dbc(\PP^1)$.  
Then by Proposition \ref{prp:2-9}, 
for the calculation 
of $\zeta_{h, \infty}(j_! R\pi^{\prime}_! \F) (t)
= \zeta_{h, \infty}(Rg_* \iota_! \F) (t) 
\in \CC (t)^*$ it suffices to calculate 
the monodromy zeta function 
$\zeta_{h \circ g}(\iota_! \F) (t) 
\in \CC (t)^*$ of $\iota_! \F 
\in \Dbc (X_{\Sigma})$ 
at each point of $g^{-1} 
(\infty ) \subset X_{\Sigma}$. We can easily 
see that $g^{-1} 
(\infty ) $  is a union of $T$-orbits. 
Let $\sigma \in \Sigma$ be a $d$-dimensional 
cone in $\Sigma$ such that $T_{\sigma} 
\subset g^{-1}(\infty )$. We choose an $n$-dimensional 
cone $\sigma_0 \in \Sigma$ such that 
$\sigma \prec \sigma_0$ and let 
$\{ \tl{u}(\sigma_0)^1, \tl{u}(\sigma_0)^2, \ldots , 
\tl{u}(\sigma_0)^n \}$ be its 
$1$-skelton. Let $\CC^n(\sigma_0)
\simeq \CC^n_y \subset X_{\Sigma}$ be 
the smooth toric variety 
defined by $\sigma_0$ as 
above. Without 
loss of generality we may assume that 
$\{ \tl{u}(\sigma_0)^1,  \ldots , 
\tl{u}(\sigma_0)^d \}$ is the $1$-skelton 
of $\sigma$. Then by the condition 
$T_{\sigma} \subset g^{-1}(\infty )$ at least 
one of the vectors $\tl{u}(\sigma_0)^1,  \ldots , 
\tl{u}(\sigma_0)^d $ is contained in 
the open half space 
\begin{equation} 
H_{<0}= \{ \tl{u}=(\tl{u}_1, \tl{u}_2, \ldots , 
\tl{u}_n) \ | \ \tl{u}_n <0 \} \subset \RR^n_{\tl{u}}. 
\end{equation} 
Since $\Sigma$ is a subdivision of $\Sigma_2$ 
we get also 
\begin{equation} 
\sigma \subset \sigma_0 \subset 
H_{ \leq 0}= \{ \tl{u}=(\tl{u}_1, \tl{u}_2, \ldots , 
\tl{u}_n) \ | \ \tl{u}_n \leq 0 \}. 
\end{equation} 
In the affine open subset 
$\CC^n(\sigma_0)
\simeq \CC^n_y$ of $X_{\Sigma}$ the 
$(n-d)$-dimensional $T$-orbit 
$T_{\sigma}$ is defined by 
\begin{equation} 
T_{\sigma}= 
\{ y \in  \CC^n(\sigma_0)
 \ | \ y_1= \cdots =y_d =0 \ \text{and} \ 
y_{d+1}, \ldots , y_n \in \CC^* \}. 
\end{equation} 
Denote the meromorphic extension of the 
Laurent polynomial $p(x,s)$ to $X_{\Sigma}$ 
simply by $p$. Then on $\CC^n(\sigma_0)
\simeq \CC^n_y$ the meromorphic 
function $p$ has the form 
\begin{equation} 
p(y)=y_1^{l_1}y_2^{l_2} \cdots 
y_n^{l_n} \times p_{\sigma_0}(y), 
\end{equation} 
where we set 
\begin{equation} 
l_i= \min_{ \tl{v} \in \tl{P}} 
\langle \tl{u}(\sigma_0)^i, \tl{v} 
\rangle \in \ZZ
\end{equation} 
for $i=1,2, \ldots , n$ and 
$p_{\sigma_0}(y)$ is a 
polynomial on $\CC^n(\sigma_0)
\simeq \CC^n_y$. 
By the non-degeneracy 
of $p(x,s)$ the complex hypersurface 
$\{ y \in  \CC^n(\sigma_0)
 \ | \ p_{\sigma_0}(y)=0 \}$ intersects 
$T_{\sigma}$ transversally. 
Let $\iota_{\sigma_0} : T \cap 
\{ p_{\sigma_0} \not= 0\} \hookrightarrow 
\CC^n(\sigma_0) \simeq \CC^n_y$  
be the inclusion. Then on $\CC^n(\sigma_0)
\simeq \CC^n_y$ we have an isomorphism 
\begin{equation} 
\iota_! \F \simeq 
( \iota_{\sigma_0})_! \LL^{\prime}, 
\end{equation} 
where $\LL^{\prime}$ is a local 
system on $T \cap 
\{ p_{\sigma_0} \not= 0\} $. 
Moreover on $\CC^n(\sigma_0)
\simeq \CC^n_y$ the meromorphic 
function $g$ has the form 
\begin{equation} 
g(y)=y_1^{l'_1}y_2^{l'_2} \cdots 
y_n^{l'_n}, 
\end{equation} 
where we set 
\begin{equation} 
l'_i= 
\langle \tl{u}(\sigma_0)^i, (0, \ldots , 0,1) 
\rangle \leq 0 
\end{equation} 
for $i=1,2, \ldots , n$.  Hence the 
function $h \circ g$ has 
the form 
\begin{equation} 
(h \circ g)(y)=y_1^{-l'_1}y_2^{-l'_2} \cdots 
y_n^{-l'_n}, 
\end{equation}
on $\CC^n(\sigma_0)
\simeq \CC^n_y$. By Proposition 
\ref{prp:5-3}, if $d=\dim \sigma \geq 2$ 
we have $\zeta_{h \circ g, y}( \iota_! \F) (t)
=1$ for any $y \in T_{\sigma}$.  
This implies that for the calculation 
of $\zeta_{h, \infty}(Rg_* \iota_! \F) (t) 
\in \CC (t)^*$ it suffices to consider 
only the cones $\sigma \in \Sigma$ of 
dimension one such that 
$T_{\sigma} \subset g^{-1}(\infty )$. From 
now on, we assume always $d=\dim 
\sigma =1$. Denote the unique 
primitive vector $\tl{u}(\sigma_0)^1
\in \ZZ^n \setminus \{ 0 \}$ on 
$\sigma$ simply by $\tl{u}(\sigma)$. 
Then by the condition $T_{\sigma} 
\subset g^{-1}(\infty )$ we 
have $\tl{u}(\sigma) \in H_{<0}$. 
On $\CC^n(\sigma_0)
\simeq \CC^n_y$ we have an isomorphism 
\begin{equation} 
\iota_! \F \simeq 
( \iota_{\sigma_0})_! \left\{ \CC \ y_1^{\delta} 
\times (y_2^{\rho_2} \cdots y_n^{\rho_n}) 
\times p_{\sigma_0}(y)^{\alpha} 
\right\} , 
\end{equation} 
where $\rho_2, \ldots , \rho_n$ 
are some complex numbers and 
we set 
\begin{equation} 
l= \min_{ \tl{v} \in \tl{P}} 
\langle \tl{u}(\sigma), \tl{v} 
\rangle \in \ZZ 
\end{equation} 
and $\delta =\alpha l + \langle 
(\beta , 0), \tl{u}(\sigma) \rangle 
\in \CC$. 
Moreover the 
meromorphic function $h \circ g:X_{\Sigma} 
\longrightarrow \PP^1$ has the form 
\begin{equation} 
(h \circ g)(y)=y_1^{-l'} 
\times (y_2^{-l'_2} \cdots y_n^{-l'_n}) 
\end{equation}
on $\CC^n(\sigma_0)$, where we set 
\begin{equation} 
l'= 
\langle \tl{u}(\sigma), (0, \ldots , 0,1) 
\rangle < 0. 
\end{equation} 
Therefore, by Proposition 
\ref{prp:5-2} and \ref{prp:5-3} 
for $y \in T_{\sigma} 
\simeq (\CC^*)^{n-1}$ 
we have 
\begin{equation}
\zeta_{h \circ g, y}(\iota_! 
\F )(t) 
=\begin{cases}
1- \exp (2\pi \sqrt{-1}\delta )t^{-l'} 
 & \text{if $y 
\notin \{ p_{\sigma_0}=0\}$ }, 
\\
1 & \text{if $y 
\in \{ p_{\sigma_0}=0\}$ }. 
\end{cases}
\end{equation}
Note that $\Gamma (\tl{P} ; \tl{u}(\sigma)) $ 
is naturally identified with the 
Newton polytope of the Laurent polynomial 
$p_{\sigma_0}|_{T_{\sigma}}$. Hence,  
if $\dim \Gamma (\tl{P} ; 
\tl{u}(\sigma)) < \dim T_{\sigma} =n-1$ 
 the Euler characteristic 
$\chi (T_{\sigma } \setminus 
\{ p_{\sigma_0}=0\} )$ 
of $T_{\sigma } \setminus 
\{ p_{\sigma_0}=0\} $ 
is zero by Bernstein-Khovanskii-Kushnirenko's 
theorem (see \cite{Khovanskii} 
and \cite{Kushnirenko} etc.). 
This implies that for the calculation of 
$\zeta_{h, \infty}(Rg_* \iota_! \F) (t) 
\in \CC (t)^*$ 
it suffices to consider only 
the $1$-dimensional 
cones $\sigma \in \Sigma$ 
such that 
$\tl{u}(\sigma) \in \ZZ^n 
\cap H_{<0}$ and $\dim \Gamma 
(\tl{P} ; \tl{u}(\sigma))=n-1$. 
In order to list up all such $1$-dimensional 
cones $\sigma \in \Sigma$, for 
$r=1,2, \ldots , k$ 
let $\tl{\Delta}_r$ be the convex hull 
of $\Delta_r \subset \RR^{n-1}_v 
\oplus \{ 0\} $ and the point 
$(a(j_0), 1)$ in $\RR^n_{\tl{v}}
= \RR^{n-1}_v \oplus \RR$. 
We also denote by $\tl{Q}$ 
the convex hull of $Q \subset 
\RR^{n-1}_v \oplus \{ 0\} $ and 
$(a(j_0), 1)$. 
Then $\tl{\Delta}_1, \tl{\Delta}_2, 
\ldots , \tl{\Delta}_k$ 
are the facets of the $n$-dimensional 
polytope $\tl{Q}$ whose inner conormal 
vectors are contained in the open half space 
$H_{<0} \subset \RR^n_{\tl{u}}$. 
For $r=1,2, \ldots , k$ let $\tl{u}^r 
\in \ZZ^n \cap H_{<0}$ be the 
unique non-zero primitive 
vector such that  
\begin{equation}
\tl{\Delta}_r=\{ \tl{v} 
\in \tl{Q} \ | \ 
\langle \tl{u}^r , \tl{v} 
\rangle = 
\min_{ \tl{w} \in \tl{Q}} 
\langle \tl{u}^r , \tl{w} 
\rangle \}. 
\end{equation}
Then by showing that 
the vector $(u^r , -d_r) \in 
\ZZ^n $ takes 
the constant value $h_r \in \ZZ$ on 
$\tl{\Delta}_r$ 
we can easily prove the following lemma. 
\begin{lemma}\label{LI}
For $r=1,2, \ldots , k$ we have 
$\tl{u}^r =(u^r , -d_r) \in 
\ZZ^n \cap H_{<0}$ and 
\begin{equation}
\min_{ \tl{v} \in \tl{P}} 
\langle \tl{u}^r , \tl{v} \rangle 
=\min_{ \tl{v} \in \tl{Q}} 
\langle \tl{u}^r , \tl{v} \rangle =h_r. 
\end{equation}
\end{lemma}
Since $d_r>0$ is the lattice distance of 
the point $(a(j_0), 0) 
\in \tl{Q}$ from $\tl{\Delta}_r$ 
by this lemma, we can easily 
see that the normalized $(n-1)$-dimensional 
volume $\Vol_{\ZZ}(\tl{\Delta}_r)$ 
of $\tl{\Delta}_r$ is 
equal to $\Vol_{\ZZ}(\Delta_r)$. 
By the definition of $\tl{P}$ 
and $\tl{u}^1, \tl{u}^2, 
\ldots , \tl{u}^k$, for 
the the calculation of 
$\zeta_{h, \infty}(Rg_* \iota_! \F) (t) 
\in \CC (t)^*$ it suffices to 
consider only the $1$-dimensional 
cones $\sigma_r=\RR_{\geq 0}\tl{u}^r 
$ ($r=1,2, \ldots , k$) in $\Sigma$. 
Hence let us consider the case where 
$\sigma=\sigma_r$ for some $r=1,2, \ldots , k$ 
and $\sigma_0$ is an $n$-dimensional cone 
in $\Sigma$ such that $\sigma \prec \sigma_0$. 
Then by Lemma \ref{LI} on 
$\CC^n(\sigma_0) \simeq \CC^n_y$ we 
have an isomorphism 
\begin{equation} 
\iota_! \F \simeq 
(\iota_{\sigma_0})_! \left\{ \CC \ y_1^{\delta_r} 
\times (y_2^{\rho_2} \cdots y_n^{\rho_n}) 
\times p_{\sigma_0}(y)^{\alpha} \right\}, 
\end{equation} 
where $\rho_2, \ldots , \rho_n$ 
are some complex numbers and 
we set $\delta_r=\alpha h_r + 
\langle \beta , u^r \rangle 
=\alpha h_r + 
\langle (\beta , 0) , \tl{u}^r \rangle 
\in \CC$ as in the 
introduction. 
By the same lemma, 
the function $h \circ g$ 
has the form $y_1^{- \langle \tl{u}^r , 
(0, \ldots , 0,1) \rangle } 
\times (y_2^{-l'_2} \cdots y_n^{-l'_n})$ 
$= y_1^{d_r} \times (y_2^{-l'_2} 
\cdots y_n^{-l'_n}) $ 
on $\CC^n(\sigma_0) \simeq \CC^n_y$. 
Then by Proposition 
\ref{prp:5-2} and \ref{prp:5-3}, 
for $y \in T_{\sigma}= T_{\sigma_r}
\simeq (\CC^*)^{n-1}$ 
we have 
\begin{equation}
\zeta_{h \circ g, y}(\iota_! 
\F )(t) 
=\begin{cases}
1- \exp (2\pi \sqrt{-1}\delta_r )t^{d_r} 
 & \text{if $y 
\notin \{ p_{\sigma_0}=0\}$ }, 
\\
1 & \text{if $y 
\in \{ p_{\sigma_0}=0\}$ }. 
\end{cases}
\end{equation}
Since the Euler characteristic 
$\chi (T_{\sigma_r} \setminus 
\{ p_{\sigma_0}=0\} )$ 
is equal to 
$(-1)^{n-1} \Vol_{\ZZ}(\tl{\Delta}_r)
=(-1)^{n-1} \Vol_{\ZZ}(\Delta_r)$ 
by Bernstein-Khovanskii-Kushnirenko's 
theorem, we obtain the desired result 
\begin{eqnarray}
\zeta_{h, \infty}(j_! R\pi_! \LL [n-1])(t) 
&=& \zeta_{h, \infty}(Rg_* 
\iota_! \F [n-1]) (t) 
\\
&=& \prod_{r=1}^k \left\{ 1-
\exp (2\pi \sqrt{-1}\delta_r) t^{d_r} 
\right\}^{\Vol_{\ZZ}(\Delta_r)}. 
\end{eqnarray}
This completes the proof. 
\end{proof}

\begin{example}(\cite[page 25-26]{S-S-T}) 
For the $3 \times 4$ integer matrix 
\begin{eqnarray}
M=(m_{ij}) =\left(  
\begin{array}{cccc}
1 & 0 & 0 & -1 \\
0 & 1 & 0 & 1 \\
0 & 0 & 1 & 1 \\
\end{array} 
\right) \in M(3,4, \ZZ )
\end{eqnarray} 
and the vector $\rho ={^t(\rho_1, \rho_2, 
\rho_3)}={^t (c-1, -a, -b)} \in \CC^3$ consider 
the following system of partial differential 
equations on $\CC^4_z$. 
\begin{gather}
\left(\sum_{j=1}^{4} m_{ij}
z_j\frac{\partial}{\partial 
z_j}-\rho_i\right)
 f(z)=0 \hspace{5mm} (1\leq i\leq 3), \\
\left\{ \prod_{\mu_j>0} \left(\frac{\partial}{\partial 
z_j}\right)^{\mu_j} -\prod_{\mu_j<0} 
\left(\frac{\partial}{\partial 
z_j}\right)^{-\mu_j} \right\} f(z)=0 
\hspace{5mm} (\mu \in {\rm Ker} M \cap 
\ZZ^{4}\setminus \{0\}). 
\end{gather}
By using the unimodular 
matrix 
\begin{eqnarray}
B =\left(  
\begin{array}{ccc}
1 & 0 & 0 \\
0 & 1 & 0  \\
1 & 1 & 1 \\
\end{array} 
\right) \in SL(3, \ZZ )
\end{eqnarray} 
let us set 
\begin{eqnarray}
\tl{A} =(a_{ij}) =BM =
\left(  
\begin{array}{cccc}
1 & 0 & 0 & -1 \\
0 & 1 & 0 & 1 \\
1 & 1 & 1 & 1 \\
\end{array} 
\right) \in M(3,4, \ZZ )
\end{eqnarray} 
and $\gamma ={^t(\gamma_1, \gamma_2, 
\gamma_3)}=B \rho ={^t(c-1, -a, 
c-a-b-1)} \in \CC^3$. Then we obtain 
an equivalent system 
\begin{gather}
\left(\sum_{j=1}^{4} a_{ij}
z_j\frac{\partial}{\partial 
z_j}-\gamma_i\right)
 f(z)=0 \hspace{5mm} (1\leq i\leq 3), \\
\left\{ \prod_{\mu_j>0} \left(\frac{\partial}{\partial 
z_j}\right)^{\mu_j} -\prod_{\mu_j<0} 
\left(\frac{\partial}{\partial 
z_j}\right)^{-\mu_j} \right\} f(z)=0 
\hspace{5mm} (\mu \in {\rm Ker} \tl{A} \cap 
\ZZ^{4}\setminus \{0\}). 
\end{gather}
on $\CC^4_z$. Since the last row of 
the matrix $\tl{A}$ is $(1,1,1,1)$, 
this is the $A$-hypergeometric 
holonomic system 
$\M_{A, \gamma}$ 
associated to 
\begin{equation}
A =\{ (1,0), (0,1), (0,0), (-1,1) 
\} \subset \ZZ^2
\end{equation} 
and $\gamma \in \CC^3$ (see the 
introduction). By Theorem \ref{TH1} 
for $j_0=1$, the 
characteristic polynomial 
$\lambda_{1}^{\infty}(t)$ of 
the $1$-st monodromy at infinity 
of the $A$-hypergeometric functions 
${\rm Sol}(\M_{A, \gamma})$ is given by 
\begin{equation}\label{monod} 
\lambda_{1}^{\infty}(t)= 
\left\{ t-\exp (2\pi \sqrt{-1}(c-a)) \right\} 
\cdot \left\{ t-\exp (2\pi 
\sqrt{-1}(c-b)) \right\} . 
\end{equation}
On the other hand, according to 
\cite[page 25-26]{S-S-T} the holomorphic 
solutions $f(z)$ to $\M_{A, \gamma}$ have the form 
\begin{equation}
f(z)=z_1^{c-1}z_2^{-a}z_3^{-b}g 
\left( 
\frac{z_1z_4}{z_2z_3} 
\right) , 
\end{equation}
where $g(x)$ satisfy the 
Gauss hypergeometric equation 
\begin{equation}
x(1-x)\frac{d^2g}{dx^2}(x)
+\{ c- (a+b+1)x \} \frac{d g}{d x}(x) 
-ab g(x)=0. 
\end{equation}
Since the characteristic exponents 
of this equation at $\infty \in \PP$ 
are $a,b \in \CC$, in this very special case 
we can check that the monodromy at infinity 
of the restriction of 
${\rm Sol}(\M_{A, \gamma})$ to a generic 
complex line $L \simeq \CC \subset \CC^4_z$ 
of the form 
\begin{equation}
L=\{ z \in \CC^4 \ | \ 
z_2=c_2, z_3=c_3, z_4=c_4 \} 
\end{equation}
is given by the formula 
\eqref{monod} .  
\end{example}

\end{document}